\documentclass{article}
\usepackage[utf8]{inputenc}
\usepackage{amsmath}
\usepackage{amsthm}
\usepackage{amssymb}
\usepackage{tikz}
\usepackage{xcolor}

\newtheorem{theorem}[equation]{Theorem}
\newtheorem{lemma}[equation]{Lemma}

\newtheorem{conjecture}[equation]{Conjecture}

\newtheorem{observation}[equation]{Observation}

\newcommand\E{\mathbb{E}}

\title{Expected number of faces in a random embedding of any graph is at most linear}

\author{Jesse Campion Loth\\[1mm]
Department of Mathematics\\
Simon Fraser University\\
Burnaby, BC, V5A 1S6, Canada
\and
Bojan Mohar\thanks{B.M.~was supported in part by the NSERC Discovery Grant R611450 (Canada) and by the Research Project N1-0218 of ARRS (Slovenia).}~\thanks{On leave from IMFM, Jadranska 19, 1000 Ljubljana}\\[1mm]
Department of Mathematics\\
Simon Fraser University\\
Burnaby, BC, V5A 1S6, Canada
}

\begin{document}
\maketitle

\begin{abstract}
A random 2-cell embedding of a given graph $G$ is obtained by choosing a random local rotation around every vertex.
We analyze the expected number of faces of such an embedding, which is equivalent to studying its average genus. In 1991, Stahl \cite{St91Short} proved that the expected number of faces in a random embedding of an arbitrary graph of order $n$ is at most $n\log(n)$. While there are many families of graphs whose expected number of faces is $\Theta(n)$, none are known where the expected number would be super-linear. This led the authors of \cite{loth2021random} to conjecture that there is a linear upper bound. In this note we confirm their conjecture by proving that for any $n$-vertex multigraph, the expected number of faces in a random 2-cell embedding is at most $2n\log(2\mu)$, where $\mu$ is the maximum edge-multiplicity. This bound is best possible up to a constant factor. 
\end{abstract}

\section{Introduction}

By an \emph{embedding} of a graph $G$ we mean a 2-cell embedding of $G$ in some orientable closed surface, and we consider two embeddings of $G$ as being the same (or \emph{equivalent}) if there is a homeomorphism between the corresponding surfaces that induces the identity isomorphism on $G$. Equivalent embeddings are considered the same, and when we speak about all embeddings of a graph, we mean all equivalence classes. It is well known that the equivalence classes of all embeddings are in bijective correspondence with \emph{rotation systems}, which are defined as the collection of local rotations at the vertices of the graph, where by a \emph{local rotation at $v$} we mean a cyclic ordering of the half-edges, (we call these darts), incident with $v$.  We refer to \cite{MT01} for more details.

It is a classical problem to study the minimum genus and maximum genus of a graph across all of its embeddings, see \cite{GT87,MT01,Wh73}. Considering the set of all 2-cell embeddings of a graph is also a viable topic. An outline of various applications of graph embeddings can be found in \cite{lando2004graphs}. In this work, we consider the problem
of the \emph{average genus} across all the different embeddings of a fixed graph.  By Euler's formula, this is equivalent to studying the average number of faces across all embeddings of a graph.  It will be more convenient to state our results in terms of the number of faces, as it better illustrates our bounds.  Formally, we consider the uniform distribution across all embeddings of a fixed graph using rotation systems, and study $\E[F]$ where $F$ is the random variable denoting the number of faces in a random embedding of the graph.

This field of study was termed \emph{random topological graph theory} by White \cite{white1994introduction}.  Stahl \cite{St91Short} gave an upper bound on $\E[F]$ by proving that $\E[F] \leq n\log n$ for any simple graph on $n$ vertices.  It was shown in \cite{loth2021random} that there are many examples of graphs with $\E[F] = \Theta(n)$: suppose a graph has maximum vertex degree $d$ and a set $\mathcal C$ of cycles, all of length at most $\ell$.  Then it is shown that 
$$\E[F] \geq \frac{2 \vert \mathcal{C} \vert}{(d-1)^\ell}.$$
In particular, if the graph has small vertex-degrees and a set of $\Theta(n)$ short cycles, then we have at least linearly many expected faces.  There are many examples of graphs of bounded degree and with linearly many short cycles. They all have $\E[F] = \Theta(n)$.

However, there are no known examples where $\E[F]$ was super-linear, and the following conjecture was proposed by Halasz, Masařík, Šámal, and the authors of this note. 

\begin{conjecture}[\cite{loth2021random}]\label{conj:1}
For every simple graph of order $n$, the expected number of faces when selecting an orientable embedding of $G$ uniformly at random is $O(n)$.\label{simple graphs conj}
\end{conjecture}

In fact, a more general conjecture from \cite{loth2021random} allowing for multiple edges of arbitrarily large multiplicity $\mu$ will be treated.

\begin{conjecture}[\cite{loth2021random}]\label{conj:2}
For every $n$-vertex multigraph $G$ with maximum edge-multiplicity $\mu \geq 2$, the expected number of faces when selecting an orientable embedding of $G$ uniformly at random is $O(n \log (\mu))$.\label{endingconj} 
\end{conjecture}

We first give some examples which show that this more general conjectured bound is tight.  We define a \emph{dipole} as the graph with $2$ vertices joined by $\mu$ edges.  Stahl first showed \cite{stahl1995average} that for the dipole on $\mu$ edges, $\E[F] \leq H_{\mu-1} + 1$ where 
$$H_\mu = 1 + \frac12 + \frac13 + \cdots + \frac1\mu$$ 
is the harmonic number.  It was later shown \cite{loth2021random} using Stanley's generating function \cite{stanley2011two} that $\E[F] = H_{\mu-1} + \left\lceil \frac{\mu}{2} \right\rceil^{-1}$.  Since $H_\mu \sim \log(\mu) + \gamma$, where $\gamma$ is the Euler-Mascheroni constant, this gives a tight example for $n=2$.  In fact Stanley's generating function may be used (see \cite{loth2021random}) to show that for any graph with one central vertex incident to all of the $\mu$ edges, $\E[F] \leq H_\mu + \frac{3}{\mu}$.

A tight example, up to a constant factor, where $n$ and $\mu$ may both tend to infinity is obtained by attaching a series of dipoles via cut edges as shown in Figure \ref{fig:dipolechain}. More precisely, consider $n/2$ dipoles, each with $\mu$ parallel edges, joined by cut edges.  Each separate dipole contains an average of at least $H_{\mu}$ faces, and joining these dipoles by cut-edges removes $n/2-1$ faces.  Therefore $\E[F] \geq \tfrac12 n (H_\mu - 1) $.

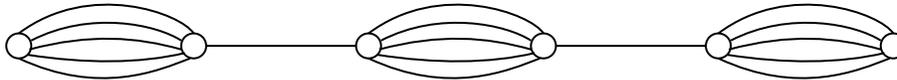
\begin{figure}
    \centering

\tikzset{every picture/.style={line width=0.75pt}} 
\begin{tikzpicture}[x=0.75pt,y=0.75pt,yscale=-0.9,xscale=0.9]

\draw   (28,42) .. controls (28,38.13) and (31.13,35) .. (35,35) .. controls (38.87,35) and (42,38.13) .. (42,42) .. controls (42,45.87) and (38.87,49) .. (35,49) .. controls (31.13,49) and (28,45.87) .. (28,42) -- cycle ;
\draw   (126,42) .. controls (126,38.13) and (129.13,35) .. (133,35) .. controls (136.87,35) and (140,38.13) .. (140,42) .. controls (140,45.87) and (136.87,49) .. (133,49) .. controls (129.13,49) and (126,45.87) .. (126,42) -- cycle ;
\draw   (224,42) .. controls (224,38.13) and (227.13,35) .. (231,35) .. controls (234.87,35) and (238,38.13) .. (238,42) .. controls (238,45.87) and (234.87,49) .. (231,49) .. controls (227.13,49) and (224,45.87) .. (224,42) -- cycle ;
\draw   (322,42) .. controls (322,38.13) and (325.13,35) .. (329,35) .. controls (332.87,35) and (336,38.13) .. (336,42) .. controls (336,45.87) and (332.87,49) .. (329,49) .. controls (325.13,49) and (322,45.87) .. (322,42) -- cycle ;
\draw   (420,42) .. controls (420,38.13) and (423.13,35) .. (427,35) .. controls (430.87,35) and (434,38.13) .. (434,42) .. controls (434,45.87) and (430.87,49) .. (427,49) .. controls (423.13,49) and (420,45.87) .. (420,42) -- cycle ;
\draw   (518,42) .. controls (518,38.13) and (521.13,35) .. (525,35) .. controls (528.87,35) and (532,38.13) .. (532,42) .. controls (532,45.87) and (528.87,49) .. (525,49) .. controls (521.13,49) and (518,45.87) .. (518,42) -- cycle ;
\draw    (35,35) .. controls (75,5) and (126,24) .. (133,35) ;
\draw    (42,39) .. controls (75,22) and (112,30) .. (126,39) ;
\draw    (41,44) .. controls (75,35) and (101,37) .. (126,44) ;
\draw    (40,47) .. controls (72,52) and (103,52) .. (128,46) ;
\draw    (35,49) .. controls (80,70) and (113,56) .. (131,49) ;
\draw    (140,42) -- (224,42) ;
\draw    (336,42) -- (420,42) ;
\draw    (231,35) .. controls (271,5) and (322,24) .. (329,35) ;
\draw    (238,39) .. controls (271,22) and (308,30) .. (322,39) ;
\draw    (237,44) .. controls (271,35) and (297,37) .. (322,44) ;
\draw    (236,47) .. controls (268,52) and (299,52) .. (324,46) ;
\draw    (231,49) .. controls (276,70) and (309,56) .. (327,49) ;
\draw    (427,35) .. controls (467,5) and (518,24) .. (525,35) ;
\draw    (434,39) .. controls (467,22) and (504,30) .. (518,39) ;
\draw    (433,44) .. controls (467,35) and (493,37) .. (518,44) ;
\draw    (432,47) .. controls (464,52) and (495,52) .. (520,46) ;
\draw    (427,49) .. controls (472,70) and (505,56) .. (523,49) ;

\end{tikzpicture}

    \caption{A chain of dipoles joined by cut edges gives a tight example for the main result of the paper.}
    \label{fig:dipolechain}
\end{figure}

Our main result confirms Conjectures \ref{conj:1} and \ref{conj:2}.  In fact we prove a more general bound in Theorem \ref{thm:main}, which allows for different edge multiplicities. The following result which implies both conjectures, is a simple corollary of it.

\begin{theorem}
    Let $G$ be a graph on $n$ vertices with maximum edge-multiplicity $\mu$, and let $F$ be the random variable for the number of faces in a random embedding of $G$.  Then we have:
    $$
    \E[F] \leq n \, (H_{2 \mu} + 1).
    $$
\end{theorem}

In the case when $G$ is a simple graph, we are able to show the better bound of $\frac{\pi^2}{6} n$ in Theorem \ref{thm:simple}.  We are unaware of any examples of simple graphs which come close to the constant in this upper bound.  A chain of triangles connected by cut edges gives an example of a graph for which $\E[F] = \frac13 n + 1$. We conjecture that this is the optimal bound.

\begin{conjecture}
    For any simple graph $G$ of order $n$, $\E[F] \leq \frac13 n + 1$.
\end{conjecture}

\section{Random embeddings}

Fix a graph $G$ on $n$ vertices with $V(G) = [n] := \{1,2,\dots,n\}$.  Let $\mu_{i,j}$ be the number of edges between vertices $i$ and $j$, where we could have $i=j$.  Let $\mu_i$ be the maximum of the multiplicities of edges incident with vertex $i$, where we count loops twice.  That is, let $\mu_i = \max (  2\mu_{i,i}, \max (\mu_{i,j} : j \neq i ))$.  Let $E_i$ be the set of darts incident with vertex $i$, let $N(i)$ be the set of vertices adjacent to $i$, and let $d_i = \vert E_i \vert$ be the degree of the vertex.  We start by outlining the random process leading to a random embedding of $G$ that we will use to prove the main result of this paper.

\bigskip

\noindent
{\bf Random Process A.}

(1) Start with the vertices $1,2,\dots,n$, with $d_i$ unlabelled darts coming out of vertex $i$ for each $i$ and fix a cyclic order of the darts around each vertex.  Over the course of the random process we will pair darts to form the edges of $G$, decreasing the number of unlabelled darts at each vertex.  When we pair a dart at vertex $i$ and a dart at vertex $j$, thus forming the edge $ij$, we label the two darts and say that we have \emph{processed} the edge $ij$.  We write $D_i$ for the number of unlabelled darts coming out of vertex $i$ at some step in the process, and write $\mu_{ij}$ for the number of unprocessed edges between $i$ and $j$ at some step.
    
(2) Repeat the following process: \\
Pick one of Option A or Option B to use at this step.\\
    Option A: Pick an edge between $i$ and $j$ which hasn't yet been chosen to process. 
    At vertices $i$ and $j$ choose one of the $D_i$ and one of the $D_j$ unlabelled darts (respectively) uniformly at random and then join them together to make an edge.\\
    Option B: Pick a dart at some vertex $i$.  Choose one of the unprocessed edges $ij$ incident with $i$ uniformly at random.  At vertex $j$, choose one of the $D_j$ unlabelled darts uniformly at random and join it with the chosen dart at $i$ to make an edge.
    
    In either option, we decrease each of $D_i, D_j, \mu_{ij}$ by one. Note that in the case of a loop ($i=j$), $D_i$ is decreased by 2.

(3) After all edges have been processed, the initial cyclic orders of darts around each vertex define a rotation system and hence an embedding of $G$.

\bigskip

By choosing the darts at step (2) of Random Process A in all possible ways, each embedding of $G$ is obtained the same number of times, and each outcome has the same probability. This shows that the process always gives an embedding of $G$ that is selected uniformly at random from the set of all embeddings. Let us also mention that the order in which the edges are processed, and whether we choose Option A or B, is not important.  These can be chosen deterministically or randomly at each step.

\begin{observation}
No matter whether we choose Option A or Option B at any step, and no matter which edge we choose when we use Option A or which dart we choose when we use Option B, at the end of the Random Process A, each embedding of $G$ is obtained with the same probability.
\end{observation}

At each step during Process A, we have a \emph{partial rotation} for which we can define \emph{(partial) faces}. The \emph{partial facial walk} around a partial face starts with an unlabelled dart, then it follows the already processed edges (maybe none) using the local rotation at vertices until we come to another unlabelled dart that is the end of this partial facial walk. See Figure \ref{fig:partial rotation}, where the partial facial walks starting at $c$ and $a$ (respectively) are outlined with thick lines.  Each dart is the beginning dart of a partial facial walk and is also the ending dart of some partial facial walk. We call a partial walk which starts and ends on opposite sides of the same unlabelled dart a \emph{bad partial facial walk}, and the corresponding face a \emph{bad partial face}.  This special type of partial face will be of special significance.  In addition to the partial facial walks, we have facial walks that use only already processed edges. These will be unchanged for the rest of the process and will be facial walks of the final embedding of $G$. We say that such a completed facial walk is \emph{closed} and is no longer considered to be a partial facial walk. Each closed facial walk became closed when we processed the last of its edges during the process. When we process the edge $ij$, there could be several pairs of a dart at $i$ and a dart at $j$ whose pairing will close a face.  If there are $k$ such pairs, we say that \emph{$k$ faces can be closed} while processing that edge.

\begin{figure}
    \centering
    \includegraphics{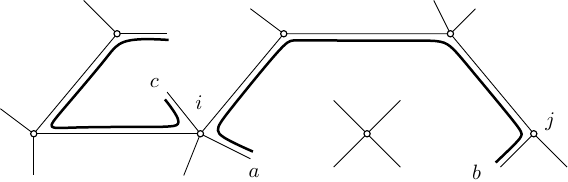}
    \caption{Partial rotation after processing five edges. Unlabelled darts are shown as short halfedges, whose local rotation is as given at the beginning, but it is not yet decided which of these will correspond to particular edges of $G$.
    Two of the partial facial walks are shown by a thick tracing line. If we are processing the edge $ij$ and choose darts $a$ and $b$ to be paired, a partial face will be closed. If we choose $c$ and $b$ instead, the two partial faces will be merged into a larger partial face, which will not be closed.}
    \label{fig:partial rotation}
\end{figure}

We now discuss a special order in which we will process the edges while generating random embeddings. 
A particular processing of edges in Process A will be termed as the \emph{greedy process}. This process is as follows:
\begin{enumerate}
    \item If the partial rotation has some bad partial faces, then pick a dart in some bad partial face to process and carry out Option B.  We prioritise picking a dart in a bad partial face which was also in a bad partial face at the previous step of Random Process A.
    \item If the partial rotation has no bad partial faces, then take Option A and choose the next edge $ij$ so that the number of faces that can be closed by processing this edge divided by $\mu_{ij}$ is minimum possible.
\end{enumerate}
An important property that we have when using the greedy process is that whenever an edge $ij$ is processed, there is a bound on the number of faces that can be closed by processing it.

\begin{lemma} \label{lem:closingfaces}
    During Random Process $A$, if the partial rotation at the start of the step has no bad partial faces then there is always an unprocessed edge $ij$, for which there are $2\mu_{ij}$ faces that can be closed by processing this edge.  Processing this edge either closes zero, one or two of the $2 \mu_{ij}$ possible closeable faces.
    
    Also, any partial rotation appearing at some step of the greedy version of Random Process $A$ has at most $2$ bad partial faces.  Each bad partial face appears in at most two consecutive steps of the greedy process.
\end{lemma}

\begin{proof}
    For the first claim, notice that the rotation system of the unlabelled darts and edges at each step is fixed. Each step of the process will join up two of the unlabelled darts into an edge.  After $\vert E \vert$ steps, we will end up with an embedding of $G$.  Recall that we defined a partial face as a face which has not yet been completed (closed) during the previous steps of the process.  The walk along a partial face starts with an unlabelled dart at some vertex. It will then alternate along edges and vertices until it eventually reaches an unlabelled dart (possibly the same one we started with), with which the partial walk ends.

At the start of the random process we have $2 \vert E \vert$ partial faces, where each partial face consists of two darts that are consecutive around the vertex in the local rotation.  At each step we join together two darts to make an edge. We claim that this always reduces the number of partial faces by two, and possibly creates one or two closed faces.  Indeed each of these darts we are joining to make an edge is the start and end of a partial face: write $f_1, f_2$ for the partial faces starting and ending respectively at one of the darts, and $f_3, f_4$ for the partial faces starting and ending respectively at the other dart as shown in Figure \ref{fig:partialfaces}. Note that $f_1$ and $f_3$ start with different darts, so they cannot be equal. Similarly, we have $f_2\ne f_4$.  There are a couple of cases:

\begin{itemize}
    \item $f_1,f_2,f_3,f_4$ are all distinct.  Then joining the two darts into an edge joins $f_1$ and $f_4$ into a partial face, and $f_2$ and $f_3$ into a partial face.
    \item $f_1 = f_4$ and $f_2 \neq f_3$, then we close the partial face $f_1=f_4$ into one completed closed face, and join $f_2$ and $f_3$ into a partial face.  The case where $f_1 \neq f_4$ and $f_2 = f_3$ is the same.
    \item $f_1 = f_4$ and $f_2 = f_3$, then we close both of these partial faces into two closed faces.
\end{itemize}
This covers all the cases. Notice that in all of the above cases we reduce the number of partial faces by two, proving the claim.
    
    This means that after $k$ edges have been processed, there are $\vert E \vert - k$ remaining unprocessed edges and $2 \vert E \vert - 2k$ partial faces.  Each partial face starts with a dart at some vertex $i$, and ends with a dart at some vertex $j$, where we may have $i=j$.  Each unprocessed edge is also associated to a pair of vertices $ij$, noting that $\sum \mu_{ij} = \vert E \vert - k$.  By the pigeonhole principle there is at least one pair $i,j$ (where we could have $i=j$) with $\mu_{ij} \geq 1$ unprocessed edges and at most $2 \mu_{ij}$ partial faces associated to it.  Hence we can always choose an edge such that processing it has at most $2 \mu_{ij}$ different faces that could be closed by processing it.

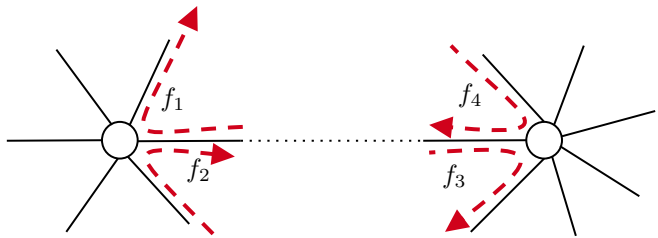
\begin{figure}
    \centering
\tikzset{every picture/.style={line width=0.75pt}} 
\begin{tikzpicture}[x=0.75pt,y=0.75pt,yscale=-1,xscale=1]

\draw   (213,79.74) .. controls (213,74.62) and (217.03,70.48) .. (222,70.48) .. controls (226.97,70.48) and (231,74.62) .. (231,79.74) .. controls (231,84.85) and (226.97,89) .. (222,89) .. controls (217.03,89) and (213,84.85) .. (213,79.74) -- cycle ;
\draw    (231,80.24) -- (283,80) ;
\draw   (427,79.74) .. controls (427,74.62) and (431.03,70.48) .. (436,70.48) .. controls (440.97,70.48) and (445,74.62) .. (445,79.74) .. controls (445,84.85) and (440.97,89) .. (436,89) .. controls (431.03,89) and (427,84.85) .. (427,79.74) -- cycle ;
\draw    (375,80) -- (427,79.74) ;
\draw    (247,29) -- (227,72) ;
\draw    (226,88) -- (257,122) ;
\draw    (405,36.48) -- (436,70.48) ;
\draw    (436,89) -- (399,126) ;
\draw [color={rgb, 255:red, 208; green, 2; blue, 27 }  ,draw opacity=1 ][line width=1.5]  [dash pattern={on 5.63pt off 4.5pt}]  (284,73) .. controls (217.02,75.96) and (225.72,88.61) .. (259.44,15.4) ;
\draw [shift={(261,12)}, rotate = 114.44] [fill={rgb, 255:red, 208; green, 2; blue, 27 }  ,fill opacity=1 ][line width=0.08]  [draw opacity=0] (11.61,-5.58) -- (0,0) -- (11.61,5.58) -- cycle    ;
\draw [color={rgb, 255:red, 208; green, 2; blue, 27 }  ,draw opacity=1 ][line width=1.5]  [dash pattern={on 5.63pt off 4.5pt}]  (269,127) .. controls (225.88,82.9) and (219.26,81.06) .. (276.42,87.59) ;
\draw [shift={(280,88)}, rotate = 186.55] [fill={rgb, 255:red, 208; green, 2; blue, 27 }  ,fill opacity=1 ][line width=0.08]  [draw opacity=0] (11.61,-5.58) -- (0,0) -- (11.61,5.58) -- cycle    ;
\draw [color={rgb, 255:red, 208; green, 2; blue, 27 }  ,draw opacity=1 ][line width=1.5]  [dash pattern={on 5.63pt off 4.5pt}]  (382.61,72.51) .. controls (455.64,80.43) and (422.32,66.3) .. (389,32) ;
\draw [shift={(378,72)}, rotate = 6.5] [fill={rgb, 255:red, 208; green, 2; blue, 27 }  ,fill opacity=1 ][line width=0.08]  [draw opacity=0] (11.61,-5.58) -- (0,0) -- (11.61,5.58) -- cycle    ;
\draw [color={rgb, 255:red, 208; green, 2; blue, 27 }  ,draw opacity=1 ][line width=1.5]  [dash pattern={on 5.63pt off 4.5pt}]  (389.34,123.53) .. controls (442.58,83.92) and (431.9,82.08) .. (378,86) ;
\draw [shift={(386,126)}, rotate = 323.62] [fill={rgb, 255:red, 208; green, 2; blue, 27 }  ,fill opacity=1 ][line width=0.08]  [draw opacity=0] (11.61,-5.58) -- (0,0) -- (11.61,5.58) -- cycle    ;
\draw    (190,34) -- (218,72) ;
\draw    (195,126) -- (221,89) ;
\draw    (165,80) -- (213,79.74) ;
\draw    (440,88) -- (452,128) ;
\draw    (444,82.74) -- (484,108) ;
\draw    (445,77.74) -- (492,65) ;
\draw    (441,72) -- (456,32) ;
\draw  [dash pattern={on 0.84pt off 2.51pt}]  (283,80) -- (375,80) ;

\draw (241,51.5) node [anchor=north west][inner sep=0.75pt]   [align=left] {$\displaystyle f_{1}$};
\draw (254,87.5) node [anchor=north west][inner sep=0.75pt]   [align=left] {$\displaystyle f_{2}$};
\draw (391,49.5) node [anchor=north west][inner sep=0.75pt]   [align=left] {$\displaystyle f_{4}$};
\draw (383,90.5) node [anchor=north west][inner sep=0.75pt]   [align=left] {$\displaystyle f_{3}$};
\end{tikzpicture}
    \caption{The situation when we have chosen two darts, and are replacing them with an edge.  The partial facial walks $f_1,f_4$ will merge, and the partial facial walks $f_2,f_3$ will merge.  This may also add one or two closed faces.}
    \label{fig:partialfaces}
\end{figure}

    For the second claim, we use induction.  Initially we may assume $G$ has no vertices of degree one, as these will not affect the final number of faces in an embedding of $G$, so we have no bad partial faces.  Then, suppose that the partial rotation we have at the start of a step has at most two bad partial faces.
    
     Case 1: It has no bad partial faces.  Then since processing the edge affects at most two faces, we can add at most $2$ bad partial faces.
     
     Case 2: It has one or two bad partial faces.  Then the greedy version of Random Process A will pair a dart in a bad partial face, removing it.  At most one other partial face will be affected by adding this edge, so we can add at most one new bad partial face.

     In either case, the number of bad partial faces in the new partial rotation is also at most two.  Also, since there is at most one dart in a bad partial face which was not processed at this step, the greedy process must process this dart at the next step.  Therefore this unlabelled dart appears in a bad partial face in at most two consecutive steps of the random process.
\end{proof}

An analysis of the greedy version of Random Process A gives our main result.  Recall that $\mu_i$ was defined as the maximum multiplicity of edges incident with vertex~$i$, counting loops twice.

\begin{theorem} \label{thm:main}
    $\E[F] \leq n + \sum_{i=1}^n H_{\mu_i}$.
\end{theorem}
\begin{proof}
    At the start of each step in our random process, we have some partial rotation.  The random process then fixes one edge into the embedding to obtain a new partial rotation.  Since each step fixes one edge, there are $|E|$ total steps.  If the partial rotation $R_k$ appears at the start of step $k$ for $k=1,\dots, |E|$, then we say that the sequence $(R_1, R_2, \dots, R_{|E|})$ occurs at this run of Random Process A.  Note that no edges have been fixed in $R_1$, and all but one edge has been fixed in $R_{|E|}$.  However there will only be one place to put the final edge, so this determines a unique embedding.  We denote by $\mathbb{P}[R_1, R_2, \dots, R_{|E|}]$ the probability that we obtain this sequence.  Similarly let $\mathbb{P}[R_k = R]$ denote the probability that we obtain $R$ at the start of step $k$.  Let $X_k$ denote the random variable for the number of faces closed at step $k$ of the random process.
    
    Now suppose the partial rotation at the start of a step is $R$.  The random process will then add an edge to this partial rotation, which will possibly close some faces.  Let $X(R)$ denote the random variable for the number of faces closed at this step of the random process.  Recall that $X(R) \in \{0,1,2\}$, as shown in Lemma \ref{lem:closingfaces}.
    
      The total probability formula gives that:
    $$
        \E[X_k] = \sum_R \mathbb{P}[R_k = R] \, \E[X_k \mid R_k = R] = \sum_R \mathbb{P}[R_k = R] \, \E[X(R)]
    $$
    where the sum runs over all possible partial rotations.
    
    Using linearity of expectation yields:
    \begin{align*}
        \E[F] &= \sum_{k=1}^{|E|} \E[X_k] = \sum_{k=1}^{|E|} \sum_R \mathbb{P}[R_k = R] \, \E[X(R)] \\
        &= \sum_{k=1}^{|E|} \sum_R \sum_{\substack{(R_1, R_2, \dots, R_{|E|}) \\ R = R_k}} \mathbb{P}[R_1, R_2, \dots, R_{|E|}] \, \E[X(R)].
    \end{align*}
    Switching the order of summation, and then taking the maximum element in the sum, gives the following:
    \begin{align*}
        \E[F] &= \sum_{(R_1, R_2, \dots, R_{|E|})} \mathbb{P}[R_1, R_2, \dots, R_{|E|}] \, \sum_{k=1}^{|E|} \E[X(R_k)] \\
        &\leq \max_{(R_1, R_2, \dots, R_{|E|})}\left\{\sum_{k=1}^{|E|} \E[X(R_k)]\right\}.
    \end{align*}
    
    Therefore we may analyse each step of the random process separately, over any fixed possible sequence of partial rotations $(R_1, R_2, \dots, R_{|E|})$.  Suppose that we are at the start of step $k$ of the random process, and we have the partial rotation $R_k$. Further suppose that $R_k$ has no bad partial faces, and that we have chosen an edge $e = ij$ to process using Option A.  Let us first suppose that $i \neq j$.  Recall that $D_i$ and $D_j$ are the number of unlabelled darts at vertices $i$ and $j$ at this step, respectively.  So, there are $D_i D_j$ choices of places to place the edge across two darts at these vertices.  However in the greedy version of Process A we choose an edge to process with only $2 \mu_{ij}$ partial faces that could be closed.  Each partial face is closed by only one choice out of the $D_i D_j$ total placements of the edge, hence the probability that we close a face is $\frac{1}{D_i D_j}$.  By Lemma \ref{lem:closingfaces} at most two faces may be closed by the same choice of edge placement, but in any case we have $\E[X(R_{k})] \leq \frac{2 \mu_{ij}}{D_i D_j}$.   Note that $\mu_{ij} \leq \min(D_i, D_j)$.  Write $D_i = D_i(R_{k}), D_j = D_j(R_{k})$ for the values of $D_i,D_j$ at this step.  Then at the next step we have $D_i(R_{k+1}) = D_i(R_k) - 1, D_j(R_{k+1}) = D_j(R_k) - 1$.

    If $R_k$ has at least one bad partial face, then recall that the greedy version of Random Process $A$ will take a dart in a bad partial face, incident with some vertex $i$.  It will then pick an unprocessed edge $ij$ incident with $i$ uniformly at random, and a dart incident with $j$ uniformly at random to pair the dart at $i$ with.  Observe that a face is closed at this step if and only if the dart we're pairing with is also in a bad partial face.  Since there is at most one other bad partial face, there is at most one choice of dart to pair with which will close a face.  Suppose that this dart in the other bad partial face is incident with vertex $j'$, if such a dart exists.  Therefore the probability we make this choice, and hence close one face, is $\E[X(R_{k})] \leq \frac{\mu_{ij'}}{D_{i}D_{j'}}$.  At the next step we reduce $D_i$ and $D_{j}$ by one.  Note we don't reduce $D_{j'}$ by one, but at the next step if $j' \neq j$, we will process the dart in a bad partial face incident with vertex $j'$.  Therefore at the following step we will reduce $D_{j'}$ by one.
    
    The case when we are processing a loop from vertex $i$ to itself is similar.  If the partial rotation has no bad partial faces, then there are $\binom{D_i}{2}$ choices of placements for the edge, and there are at most $2 \mu_{ii}$ partial faces which may be closed.  Therefore by the same reasoning as in the previous case, we have $\E[X(R_k)] \leq \frac{2\mu_{ii}}{D_i (D_i - 1)/2} = \frac{4\mu_{ii}}{D_i (D_i - 1)}$.  At the next step we have $D_i(R_{k+1}) = D_i(R_k) - 2$.  A similar reasoning holds for the case when we process a dart incident with a bad partial face using Option B.

    Fix some vertex $i$.  Over the sequence of partial rotations $(R_1, R_2, \dots, R_{|E|})$, let $R_{k_1}, \dots, R_{k_c}$ be the partial rotations for which an edge at vertex $i$ is processed, where $c$ is equal to the number of edges incident with vertex $i$, counting each loop only once.  We have that $D_i(R_{k_1}) = d_i$.  If we are in option A of the random process, and the first edge processed at vertex $i$ was not a loop, then $D_i(R_{k_2}) = d_i - 1$.  If it was a loop then $D_i(R_{k_2}) = d_i-2$.  If we are in option B of the random process, then we could have $D_i(R_{k_2}) = d_i - 1$ or $D_i(R_{k_2}) = d_i$.  However if $D_i(R_{k_2}) = d_i$ then necessarily $D_i(R_{k_3}) = d_i-1$.  Therefore the values of $D_i$ decrease at each of these steps until $D_i(R_{k_c}) = 1$ or $2$, then these remaining darts are processing at this step.

    Initially we have that $D_i = d_i$ and $D_j = d_j$.  For $i \neq j$, $\mu_{ij} \leq \min(\mu_i,\mu_j)$, and $2\mu_{ii} \leq \mu_i$.  When we process a non-loop edge using Option A, $\E[X(R_k)]$ is bounded by some $\frac{2 \mu_{ij}}{D_i D_j}$, and $D_i, D_j$ and $\mu_{ij}$ all decrease by one.  When we process a loop edge using Option A, $\E[X(R_k)]$ is bounded by $\frac{4 \mu_{ii}}{D_i(D_i-1)}$, $D_i$ decreases by two and $\mu_{ii}$ decreases by one.  When we process a dart using Option B, $\E[X(R_k)]$ is bounded by $\frac{\mu_{ij'}}{D_iD_{j'}}$ for some $j'$.  For some $j$, $D_i, D_{j}$ decrease by one and $\mu_{ij}$ decreases by one.  Also when $i \neq j$, $\mu_{ij}$ is the number of unprocessed edges between vertices $i$ and $j$, so by definition we have that $\mu_{ij} \leq \min(D_i,D_j,\mu_i, \mu_j)$.  Similarly, we have that $2\mu_{ii} \leq \min(D_i,\mu_i)$.

    If at step $k$ we process an edge $e = ij$, then write $D_i(e) = D_i(R_k), D_j(e) = D_j(R_k)$.  For this sequence of partial rotations $(R_1, R_2, \dots, R_{|E|})$, write $E_A, E_B$ for the set of edges processed under Options A and B respectively.  Then we have:
    \begin{align*}
       \sum_{k=1}^{|E|} \E[X(R_k)] \leq& \sum_{\substack{(i,j) \in E_A \\ i \neq j}} \frac{2 \min(D_i(e),D_j(e),\mu_i, \mu_j)}{D_i(e) D_j(e)} + \sum_{\substack{(i,j) \in E_A \\ i = j}} \frac{2 \min(D_i(e), \mu_i)}{D_i(e) (D_i(e) - 1)} \\
       +& \sum_{(i,j) \in E_B} \frac{ \min(D_i(e),D_{j'}(e),\mu_i, \mu_j)}{D_i(e) D_{j'}(e)}.
    \end{align*}
      We first note that for any $a,b > 0$:
    \begin{align*}
        \frac{2 \min(a,b,\mu_i, \mu_j)}{ab} &\leq \frac{\min(a,b,\mu_i, \mu_j)}{a^2} + \frac{\min(a,b,\mu_i, \mu_j)}{b^2} \\
        &\leq \frac{\min(a, \mu_i)}{a^2} + \frac{\min(b,\mu_j)}{b^2}.
    \end{align*}

\noindent This means we can rewrite the expectation as:
\begin{align*}
       \sum_{k=1}^{|E|} \E[X(R_k)] &\leq \sum_{\substack{(i,j) \in E_A \\ i \neq j}} \left( \frac{\min(D_i(e), \mu_i)}{D_i(e)^2} + \frac{\min(D_j(e), \mu_j)}{D_j(e)^2} \right)  \\
       &+ \sum_{\substack{(i,j) \in E_A \\ i = j}} \left( \frac{\min(D_i(e), \mu_i)}{D_i(e)^2} + \frac{\min(D_i(e), \mu_i)}{(D_i(e)-1)^2} \right) \\
       &+ \frac{1}{2} \sum_{(i,j) \in E_B} \left( \frac{\min(D_i(e), \mu_i)}{D_i(e)^2} + \frac{\min(D_{j'}(e), \mu_j)}{D_{j'}(e)^2} \right).
    \end{align*}

Now fix some vertex $i$.  For each non-loop edge $e \in E_A$ incident with $i$ we obtain a term of $\frac{\min(D_i(e), \mu_i)}{D_i(e)^2}$ in the preceding sum.    For each loop in $E_A$ incident with vertex $i$, we obtain a term of $\frac{\min(D_i(e), \mu_i)}{D_i(e)^2} + \frac{\min(D_i(e), \mu_i)}{(D_i(e)-1)^2}$.  Recall that when we process an edge incident with $i$, in Option A, we reduce $D_i$ by one if the edge is not a loop, and by two if the edge is a loop.  When we process an edge in $E_B$ incident with $i$, we obtain a term of $\frac{\min(D_i(e), \mu_i)}{2D_i(e)^2}$.  If we don't decrease $D_i$ at this step, then we do reduce $D_i$ at the following step and obtain another term of $\frac{\min(D_i(e), \mu_i)}{2D_i(e)^2}$.   This means that in the whole sum, $D_i$ appears (as some $D_i = D_i(e)$) for each of the values in $\{1,2,\dots,d_i\}$ in at most one term of the form $\frac{\min(D_i(e), \mu_i)}{D_i(e)^2}$.  Also recall that by the preceding arguments, it is enough to bound $\sum_{k=1}^{\vert E \vert} \E[X(R_k)]$ for an arbitrary $(R_1, R_2, \dots, R_{|E|})$ in order to bound $\E[F]$.  Therefore we may bound the expectation as:
\begin{align*}
    \E[F] \leq \sum_{i=1}^n \sum_{t=1}^{d_i} \frac{\min(t,\mu_i)}{t^2} < \sum_{i=1}^n \sum_{t\geq 1} \frac{\min(t,\mu_i)}{t^2}.
\end{align*}
For each vertex $i$ we obtain a sum of terms of the form:
$$
    \sum_{t \geq 1} \frac{\min(t, \mu_i)}{t^2} \leq \sum_{t=1}^{\mu_i} \frac{1}{t} + \sum_{t \geq \mu_i + 1} \frac{\mu_i}{t^2} = H_{\mu_i} + \mu_i \sum_{t \geq \mu_i + 1} \frac{1}{t^2}.
$$
    Note that we have:
    $$
      \mu_i \sum_{t \geq \mu_i + 1} \frac{1}{t^2} 
        < \mu_i \int_{\mu_i}^{\infty} x^{-2} dx = 1.
    $$
    Therefore the total contribution from all the vertices is bounded by:
    $$
        \E[F] < \sum_{i=1}^n \left( H_{\mu_i} + 1\right) = \sum_{i=1}^n H_{\mu_i} + n.
    $$
    This gives the required bound.
\end{proof}
    
If the graph is simple, then every $\mu_i$ is equal to 1. In this case we can obtain a slightly better upper bound.

\begin{theorem}\label{thm:simple}
  If $G$ is a simple graph of order $n$, then
  $$ \E[F] < \frac{\pi^2}{6}\, n. $$
\end{theorem}
\begin{proof}
    From the proof of the general case, we have the following sum as an upper bound:
    $$
       \E[F] \leq \sum_{e = ij} \frac{2}{ab} \leq \sum_{e = ij}\left( \frac{1}{a^2} + \frac{1}{b^2} \right) ,
    $$
    where the pairs $\{a,b\}$ of the summands exhaust the multiset 
    $$\Delta = \{1,2, \dots, d_1, 1,2, \dots, d_2, \dots, 1,2,\dots, d_n\}.$$ 
    Each term of $1/a^2$ appears at most $n$ times for each $a \geq 1$, so we obtain the upper bound:
    $$
        \E[F] < n \sum_{a \geq 1} \frac{1}{a^2} = \frac{\pi^2}{6} n .
    $$
\end{proof}

\section*{Acknowledgements}

The authors would like to thank Kevin Halasz, Tomáš Masařík and Robert Šámal for helpful discussions on the topic.

\bibliographystyle{plain}
\bibliography{expectedfaceslinear}

\begin{thebibliography}{1}

\bibitem{loth2021random}
Jesse Campion~Loth, Kevin Halasz, Tom{\'a}{\v{s}} Masa{\v{r}}{\'\i}k, Bojan
  Mohar, and Robert {\v{S}}{\'a}mal.
\newblock Random 2-cell embeddings of multistars.
\newblock {\em Proceedings of the American Mathematical Society}, 2022.

\bibitem{GT87}
Jonathan~L Gross and Thomas~W Tucker.
\newblock {\em Topological graph theory}.
\newblock Wiley-Interscience, New York, 1987.

\bibitem{lando2004graphs}
Sergei~K Lando, Alexander~K Zvonkin, and Don~Bernard Zagier.
\newblock {\em Graphs on surfaces and their applications}, volume~75.
\newblock Springer, 2004.

\bibitem{MT01}
Bojan Mohar and Carsten Thomassen.
\newblock {\em Graphs on surfaces}.
\newblock Johns Hopkins Studies in the Mathematical Sciences. Johns Hopkins
  University Press, Baltimore, MD, 2001.

\bibitem{St91Short}
Saul Stahl.
\newblock An upper bound for the average number of regions.
\newblock {\em J. Combin. Theory Ser. B}, 52(2):219--221, 1991.

\bibitem{stahl1995average}
Saul Stahl.
\newblock On the average genus of the random graph.
\newblock {\em Journal of Graph Theory}, 20(1):1--18, 1995.

\bibitem{stanley2011two}
Richard~P Stanley.
\newblock Two enumerative results on cycles of permutations.
\newblock {\em European Journal of Combinatorics}, 32(6):937--943, 2011.

\bibitem{Wh73}
Arthur~T White.
\newblock {\em Graphs, groups and surfaces}.
\newblock North-Holland Publishing Co., Amsterdam-London; American Elsevier
  Publishing Co., Inc., New York, 1973.
\newblock North-Holland Mathematics Studies, No. 8.

\bibitem{white1994introduction}
Arthur~T White.
\newblock An introduction to random topological graph theory.
\newblock {\em Combinatorics, Probability and Computing}, 3(4):545--555, 1994.

\end{thebibliography}

\end{document}